\documentclass[12pt,a4paper]{article}

\pdfoutput=1

\title{Continuity of Convex Set-valued Maps and a Fundamental Duality Formula for Set-valued Optimization}
\author{Frank Heyde\thanks{University of Graz, Institute of Mathematics and Scientific Computing, Heinrichstr. 36, A-8010 Graz, Austria ({\tt frank.heyde@uni-graz.at}).} \and Carola Schrage\thanks{\tt carolaschrage@gmail.com}}
\usepackage{amsmath, amssymb,graphicx}
\usepackage{amsthm}
\usepackage[arrow, matrix, curve]{xy}
\usepackage{enumerate}

\usepackage{color}

\newcommand{\norm}  [1]{\ensuremath{\left  \|       #1  \right \|       }}

\newcommand{\sqb}   [1]{\ensuremath{\left [        #1  \right ]        }}

\newcommand{\cb}    [1]{\ensuremath{\left  \{      #1  \right \}       }}

\newcommand{\of}    [1]{\ensuremath{\left (        #1  \right )        }}

\newcommand{\abs  } [1]{\ensuremath{\left  |       #1  \right |        }}

\newcommand{\st} {\ensuremath{\;|\;}}

\newcommand{\epi} {{\rm epi \,}}

\newcommand{\cl}  {{\rm cl  \,}}

\newcommand{\Int} {{\rm int \,}}

\newcommand{\cone}{{\rm cone\,}}
\newcommand{\gr}  {{\rm gr\,}}

\newcommand{\dom}{{\rm dom\,}}

\newcommand{\V}{\mathcal{V}}

\newcommand{\A}{\mathcal{A}}

\newcommand{\F}{\mathcal{F}}

\renewcommand{\P}{\mathcal{P}}

\newcommand{\NN}{\mathcal{N}}

\newcommand{\R}{\mathrm{I\negthinspace R}}
\newcommand{\N}{\mathrm{I\negthinspace N}}

\renewcommand{\subset}{\subseteq}

\definecolor{color0}{gray}{.50}
\definecolor{color1}{rgb}{0,.2,.8}
\definecolor{color2}{rgb}{1,.2,0}
\definecolor{color3}{rgb}{0,.6,0}

\newcommand{\gre}{\color{color3}}

\newtheorem{theorem}{Theorem}[section]

\newtheorem{remark}[theorem]{Remark}
\newtheorem{lemma}[theorem]{Lemma}
\newtheorem{proposition}[theorem]{Proposition}

\newtheorem{example}[theorem]{Example}
\theoremstyle{definition}
\newtheorem{definition}[theorem]{Definition}

\newcommand{\OLR}{\overline{\mathrm{I\negthinspace R}}}

\begin{document}
%\title{Continuity Concepts for Set-valued Functions and a Fundamental Duality Formula for Set-valued Optimization}
%
%\author[KFU]{Frank Heyde\corref{cor1}}
%\ead{frank.heyde@uni-graz.at}
%
%\author{Carola Schrage}
%\ead{carolaschrage@gmail.com}
%\address[KFU]{University of Graz, Institute of Mathematics and Scientific Computing, Heinrichstr. 36, A-8010 Graz, Austria}
%%\address[MLU]{MLU Halle--Wittenberg, Institute of Mathematics, 06099 Halle (Saale), Germany}
%
%
%\begin{abstract}
%Over the past years a theory of conjugate duality for set-valued functions that map into the set of upper closed subsets of a preordered topological vector space was developed. For scalar duality theory, continuity of convex functions plays an important role. For set-valued maps different notions of continuity exist. We will compare the most prevalent ones in the special case that the image space is the set of upper closed subsets of a preordered topological vector space and analyze which of the results can be conveyed from the extended real-valued case.
%
%Moreover, we present a fundamental duality formula for set-valued optimization, using the weakest of the continuity concepts under co�nsideration for a regularity condition.
%\end{abstract}
%
%\begin{keyword}
%set-valued map \sep upper closed sets \sep continuity \sep semicontinuous function \sep convex function \sep Legendre-Fenchel conjugate \sep fundamental duality formula
%\end{keyword}

\maketitle

\begin{abstract}
Over the past years a theory of conjugate duality for set-valued functions that map into the set of upper closed subsets of a preordered topological vector space was developed. For scalar duality theory, continuity of convex functions plays an important role. For set-valued maps different notions of continuity exist. We will compare the most prevalent ones in the special case that the image space is the set of upper closed subsets of a preordered topological vector space and analyze which of the results can be conveyed from the extended real-valued case.

Moreover, we present a fundamental duality formula for set-valued optimization, using the weakest of the continuity concepts under consideration for a regularity condition.\\[1em]
\end{abstract}

\section{Introduction}
Recently, a new theory of conjugate duality was developed for set-valued functions mapping into the space $\F(Z,C)$ of upper closed subsets of a topological vector space $Z$ that is preordered by a closed convex cone $C$ (see, e.g., the papers of Hamel \cite{Hamel09}, \cite{Hamel11}, Schrage \cite{DissSchrage}, \cite{Schrage11} or the book of L\"ohne \cite{BookLoehne}).

There are two main applications for functions mapping into $\F(Z,C)$ which we would like to point out. The first one comes from Mathematical Finance. When measuring the risk of investments in several markets (or assets) under the presence of transaction costs between the markets (or assets), it turned out to be appropriate to use set-valued functions mapping into $\F(Z,C)$, where $Z$ is the space of eligible markets (or assets) and $C$ is the solvency cone representing the exchange rates and transaction costs. Details can be found in \cite{HH2010} and \cite{HHR2011}.
The second application concerns vector optimization problems and a duality theory for them. As we describe in Section \ref{sec:BCN}, the original preordered vector space $(Z,\le_C)$, which is in general no complete lattice, can be embedded into the complete lattice $(\F(Z,C),\supseteq)$. It proved to be advantageous in several aspects to consider the $\F(Z,C)$-valued optimization problem rather than the original vector optimization problem. One aspect is the possibility to use a solution concept based on the attainment of the infimum as introduced in \cite{HL2011} and another aspect is that the existence of infimum and supremum allows the definition of the conjugates in the style of the scalar case. It seems that the duality theory developed in \cite{Hamel09}, \cite{Hamel11} and \cite{DissSchrage} for $\F(Z,C)$-valued functions is the most promising for vector-valued, i.e., $(Z,\le_C)$-valued functions. 
A third application deals with epigraphical multifunctions. If $X$ is a locally convex topologiacal vector space with topological dual $X^*$ and $\Gamma(X)$ the space of all proper lower semicontinuous convex extended real-valued functions on $X$ then the function $f \mapsto \epi f$ is a set-valued function from $\Gamma(X)$ to $\F(X\times\R,\cb{0}\times\R_+)$ and $f\mapsto \epi f^*$ maps from $\Gamma(X)$ to $\F(X^*\times\R,\cb{0}\times\R_+)$. If $\Gamma(X)$ is provided with Joly's topology (see Joly \cite{Joly}) then both epigraphical multifunctions become lower continuous in the sense of Definition \ref{def:cont} below.

A couple of set-valued variants of the Fenchel-Rockafellar duality theorem have been formulated in the context of $\F(Z,C)$-valued convex analysis (\cite[Theorem 2]{Hamel11}, \cite[Theorem 4.2.9]{DissSchrage}, \cite[Theorem 5.5]{Schrage11} and \cite[Theorem 3.16]{BookLoehne}). Despite of some differences in the definition of the conjugates, the conclusions of the statements in the above mentioned theorems are equivalent. However, different regularity conditions, which can be considered as generalizations of the classic condition that one of the functions involved is continuous in one point of the domain, are used as assumptions. 
In fact, continuity, upper semicontinuity, local boundedness from above on the interior of the domain and nonemptiness of the interior of the epigraph are all equivalent properties for a convex extended real-valued function.
The purpose of the present article is to analyze these relations in the set-valued case and to come up with a preferably weak regularity condition.

This paper is organized as follows. In 
Section \ref{sec:BCN} we introduce the basic concepts and notation for set-valued functions and their conjugates that are used in this paper. Section \ref{sec:CCC} deals with the various semicontinuity notions for set-valued functions and their relationships. Finally, we present a set-valued variant of the fundamental duality formula using the weakest of the upper semicontinuity notions under consideration as a regularity condition in Section \ref{sec:FDF}.

\section{Basic concepts and notation}\label{sec:BCN}
In this section we introduce the basic concepts and notation in the context of set-valued functions and their conjugates. Details can be found, e.g., in \cite{Hamel09} and \cite{DissSchrage}.

Unless stated otherwise, throughout this article $X$ and $Z$ are topological vector spaces with topological dual spaces $X^*$ and $Z^*$, respectively, 
$\NN(x_0)$ and $\V(z_0)$ denote the systems of neighborhoods of the points $x_0$ in $X$ and $z_0$ in $Z$.
On $Z$, a preorder $\le_C$ is generated by a nonempty
closed convex cone $C\subseteq Z$, setting $z_1\leq_C z_2$ iff $z_2-z_1 \in C$.
For details about ordered topological vector spaces the reader is referred e.g. to \cite{Peressini} or \cite{WongNg}.

The negative dual cone of $C$ is denoted by
\[
C^-:=\cb{z^*\in Z^* \st z^*(z)\le 0 \text{ for all }z\in C}.
\]
Throughout the paper we assume that $C^-\neq\cb{0}$. When $X$ is locally convex then $C^-\neq\cb{0}$ is satisfied if and only if $C\neq Z$.

In order to derive a satisfactory duality theory for vector optimization problems, it turned out to be useful to embed $(Z,\le_C)$ into a suitable subset of the power set $\P(Z)$ of $Z$ (including the empty set). We consider the collection of upper closed subsets of $Z$, defined by
\[
\F(Z,C)=\cb{A\subset Z \st A=\cl (A+C)},
\] 
the set of all closed subsets of $Z$ whose recession cone contains $C$ as a subset but need not be equal to $C$. 
Note that in \cite{Hamel09} and \cite{DissSchrage} the collection of upper closed sets is denoted by $\P^t_C(Z)$, but we prefer the notation $\F$ that is used in \cite{HL2011} and \cite{BookLoehne}. 

%For convex functions also 
%\[
%\G(Z,C)=\cb{A\subset Z \st A=\cl\co (A+C)},
%\]
%the collection of upper closed and convex subsets of $Z$, plays a role. Obviously, $\G(Z,C)\subseteq\F(Z,C)$.

The preorder $\le_C$ on $Z$ is extended to the power set $\P(Z)$ by defining
\[
A \preccurlyeq_C B \quad\Leftrightarrow\quad \forall b\in B\; \exists a\in A : a\le_C b \quad\Leftrightarrow\quad B\subseteq A+C.% \quad\Leftrightarrow\quad A\supseteq B.
\]
The set $\F(Z,C)$ is partially ordered by $\preccurlyeq_C$ and for all $A,B\in\F(Z,C)$ it holds $A\preccurlyeq_C B$, iff $B\subseteq A$.
The preordered space $(Z,\le_C)$ can be embedded into the partially ordered space
$(\F(Z,C),\supseteq)$ by the map $z\mapsto \cb{z}+C$.
Moreover,  $(\F(Z,C),\supseteq)$ is a complete lattice whereas $(Z,\le_C)$ is in general not. If $\A\subseteq \F(Z,C)$ then
the infimum and supremum of $\A$ in $(\F(Z,C),\supseteq)$ are given by
\begin{align*}
&\inf \A=\cl\bigcup_{A\in\A} A, \qquad \sup \A = \bigcap_{A\in \A} A.
\end{align*}
The greatest element of $(\F(Z,C),\supseteq)$ is $\emptyset$ and the least element is $Z$.

For functions $f:X\to \P(Z)$ the graph is defined as
\begin{align*}
\gr f &=\cb{(x,z)\in X\times Z \st z\in f(x)}%\\
%\epi f &= \cb{(x,z)\in X\times Z \st z\in f(x)+C}.
\end{align*}
and the domain of $f$ is defined as
\[
\dom f =\cb{x\in X \st f(x)\neq \emptyset}.
\]

A function $f:X\to \P(Z)$ is called $C$--convex, iff %$\epi f$ is a convex set.
\[
\forall x_1,x_2\in X,\; \forall t\in\of{0,1}:\quad f(tx_1+(1-t)x_2)\preccurlyeq_C tf(x_1)+(1-t)f(x_2),  
\]
compare e.g. \cite[Definition 14.7]{Jahn04}. 
If $f:X\to\F(Z,C)$ is $C$--convex then it is convex-valued.
It is easily seen that if $f:X\to\F(Z,C)$ is $C$--convex, then $\gr f$ is a convex set and vice versa. For simplicity, we refer to $C$--convex functions mapping into $\F(Z,C)$ as convex functions.

In \cite[Definition 5]{Hamel09}, the \textit{(negative) conjugate} of a function $f:X\to \F(Z,C)$ is the function $-f^*:X^*\times C^-\setminus\cb{0}\to \F(Z,C)$ defined by
\begin{align}
(-f^*)(x^*,z^*) = \cl\bigcup\limits_{x\in X} \of{f(x)+S_{(x^*,z^*)}(-x)}
\end{align}
where $S_{(x^*,z^*)}:X\to\F(Z,C)$, defined by
\begin{align}
S_{(x^*,z^*)}(x)=\cb{z\in Z\colon x^*(x)+z^*(z)\leq 0},
\end{align}
serves as set-valued replacement of linear functionals. Compare \cite[Section 3]{Hamel09}.

Following \cite[Definition 3.1.2]{DissSchrage} we define the scalarization $\varphi_{(f,z^*)}:X\to\OLR$ of $f:X\to\F(Z,C)$ in direction $z^*\in Z^*$ by
\[
\varphi_{(f,z^*)}(x)=\inf_{z\in f(x)} -z^*(z)=-\sup_{z\in f(x)}z^*(z).
\]
So $\varphi_{(f,z^*)}(x)$ is the negative of the support functional to the set $f(x)$ in direction $z^*$. The function $f$ is convex, iff for all $z^*\in C^-\setminus\cb{0}$ the scalarizations $\varphi_{(f,z^*)}:X\to\OLR$ are convex as well. 

Whenever $Z$ is a Hausdorff locally convex space, a set-valued function $f:X\to\F(Z,C)$ with convex values is uniquely characterized by its family of scalarizations with $z^*\in C^-\setminus\cb{0}$. By a separation argument in $Z$, we have
\[
f(x)=\bigcap_{z^*\in C^-\setminus\cb{0}}\cb{z\in Z\st \varphi_{(f,z^*)}(x)\leq -z^*(z)}
\]
for all $x\in X$. Also the conjugate can be expressed by the conjugate of the corresponding scalarization by virtue of
\begin{multline*}
\forall (x^*,z^*)\in X^*\times C^-\setminus\cb{0}:\\
(-f^*)(x^*,z^*) = \cb{z\in Z\st -(\varphi_{(f,z^*)})^*(x^*)\leq -z^*(z)},
\end{multline*}
compare \cite[Lemma 1]{Hamel09}.
\begin{remark} Schrage \cite{DissSchrage, Schrage11} has also defined a positive conjugate $f^*$, but this requires the introduction of a suitable difference of sets. In order to avoid that, we use the negative conjugate in this paper.
\end{remark}
%%%%%%%%%%%%%%%%%%%%%%%%%%%%%%%%%%%%%%%%%%%%%%%%%%%%%%%%%%%%%%%%%%%%%%%%%%%%%%%%%%%%%%%%
\section{Comparison of continuity concepts}\label{sec:CCC}

In this section we will analyze the relations between several existing semicontinuity concepts for set-valued functions mapping into $\F(Z,C)$ in general and, in particular, for convex functions. Moreover, we will work out a suitable regularity condition for strong duality in set-valued optimization.

Before we deal with set-valued functions, we briefly recall the extended real-valued case. For a convex extended real-valued function continuity, upper semicontinuity, local boundedness from above at one point in the interior of the domain and nonemptiness of the interior of the epigraph are all equivalent and each of these properties implies lower semicontinuity.
Usually, a separation argument, which is true under the assumption that the epigraph has a nonempty interior, is used for proving strong duality in the scalar case. This is the assumption that is actually needed in the proof of strong duality. However, the equivalent property of continuity at one point in the interior of the domain is often chosen as regularity condition since it is more handy. But how about the set-valued case? Is continuity of a $\F$-valued function still equivalent to the nonemptiness of the interior of the epigraph and, as there are several different continuity concepts for set-valued functions, for which kind of continuity? The answer to these questions will be given below.
%\[
%\begin{xy} \xymatrix{
%\Int (\epi f)\neq \emptyset\quad \ar@2{<->}[rr] & &
%\quad\text{ l.b.a. }\quad \ar@<5pt>@2{<-}[rr] & &
%\quad\text{ u.s.c. }\quad \ar@<5pt>@2{<-}[ll]^{f\text{ convex}\quad} & &
%\quad\text{ l.s.c. }  \ar@2{<-}[ll]^{f\text{ convex}}
%}\end{xy}
%\]

%Throughout this section, $X$ is a topological space, $Z$ is a topological vector space, $C\subseteq Z$ is a convex cone and $f$ is a function from $X$ to $\F(Z,C)$ unless otherwise stated. 

First, we turn toward the classic semicontinuity notions for set-valued maps. There is a vast amount of literature dealing with these concepts, e.g., \cite{AliprantisBorder}, \cite{AubinCellina}, \cite{AubinFrankowska}, \cite{Beer}, \cite{Berge}, \cite{Deimling}, \cite{GRTZ}, \cite{RoWe98} to mention but a few. We will base our presentation on \cite[Sections 2.5 and 2.6]{GRTZ}, which furnishes a rather comprehensive treatment of these notions and their relations.

\begin{definition}\label{def:cont}
(i) %Let $X$ and $Z$ be topological spaces. 
A function $f: X \to \P(Z)$ is called {\em upper continuous} (u.c.) at a point $x_0\in X$ iff for any open set $V$ in $Z$ with $f(x_0) \subseteq V$ there is a neighborhood $U$ of $x_0$ such that $f(x)\subseteq V$ for all $x\in U$.

(ii) %Let $X$ and $Z$ be topological spaces. 
A function $f: X \to \P(Z)$ is called {\em lower continuous} (l.c.) at a point $x_0\in X$ iff for any $z_0\in f(x_0)$ and any $V\in \V(z_0)$ there is a neighborhood $U$ of $x_0$ such that $f(x) \cap V \neq \emptyset$ for all $x\in U$.

%(iii) Let $X$ and $Z$ be topological spaces. A function $f: X \to \P(Z)$ is called {\em lower continuous} (l.c.) at $(x_0,z_0)\in X\times Z$ iff for any neighborhood $V$ of $z_0$ in $Z$ there is a neighborhood $U$ of $x_0$ such that $f(x) \cap V \neq \emptyset$ for all $x\in U$.

(iii) %Let $X$ be a topological space and $Z$ be a topological vector space. 
A function $f: X \to \P(Z)$ is called {\em Hausdorff upper continuous} (H-u.c.) at a point $x_0\in X$ iff for any neighborhood $V$ of $0$ in $Z$ there is a neighborhood $U$ of $x_0$ such that $f(x)\subseteq f(x_0)+ V$ for all $x\in U$.

(iv) %Let $X$ be a topological space and $Z$ be a topological vector space. 
A function $f: X \to \P(Z)$ is called {\em Hausdorff lower continuous} (H-l.c.) at a point $x_0\in X$ iff for any neighborhood $V$ of $0$ in $Z$ there is a neighborhood $U$ of $x_0$ such that $f(x_0)\subseteq f(x)+V $ for all $x\in U$.

%(v) %Let $X$ be a topological space and $Z$ be a topological vector space. 
%A function $f: X \to \P(Z)$ is called {\em $B$-bounded above} ($B$-b.a.) on a set $M\subseteq X$ iff there is a {\color{red} define: }topologically bounded subset $B$ of $Z$ such that $f(x)\cap B\neq \emptyset$ for all $x\in M$.
\end{definition}

\begin{remark}
Although we have stated in Section \ref{sec:BCN} that $X$ should be a vector space, the linear structure of $X$ is no requirement for the preceding concepts. Likewise, in the remainder of this section each statement that does not assume the function $f$ to be convex ($C$--convex) is also true in the more general situation when $X$ is merely a topological space.
\end{remark}

\begin{remark}
Note that the definitions above apply to any $x_0\in X$. $x_0$ does not need to be in the domain of $f$. However, it is easy to see that $f$ is lower (Hausdorff) continuous at $x_0$ by force if $x_0\not\in \dom f$, and $x_0\in \Int(X\setminus \dom f)$ is necessary and sufficient for $f$ being upper (Hausdorff) continuous at $x_0\not\in \dom f$. 
\end{remark}

\begin{remark}
The notation for the above concepts varies in the literature.
Upper and lower continuity are often referred to as upper and lower semicontinuity
(\cite{AubinCellina}, \cite{AubinFrankowska}, \cite{Beer}, \cite{Berge}, \cite{Deimling}). However, we will stick to the notation from \cite{GRTZ} in order to highlight the structural differences to upper and lower semicontinuity in a lattice sense that will be considered later (see Definition \ref{def:latticeprop}). In fact, lower continuity of an $\F(Z,C)$-valued function corresponds to upper semicontinuity in the scalar case.  Compare also Remark \ref{rem:contvssemicont}.

In \cite{AliprantisBorder} the notation of upper and lower hemicontinuity is used instead of upper and lower continuity, whereas upper hemicontinuity in \cite{AubinCellina} and \cite{AubinFrankowska} means lower semicontinuity of all scalarizations. Moreover, in normed spaces Hausdorff continuity is also referred to as $\varepsilon-\delta-$semicontinuity (\cite{Deimling}) or semicontinuity in the $\varepsilon-$sense (\cite{AubinCellina}).

Other notations that occur in the literature are inner and outer semicontinuity (\cite{Beer}, \cite{RoWe98}) and closedness (\cite{GRTZ}). Whereas inner semicontinuity coincides with lower continuity, outer semicontinuity and closedness are the same, and for $\F(Z,C)$-valued functions they coincide with lower lattice-semicontinuity, which is defined later (see Definition \ref{def:latticeprop}, Remark \ref{rem:LatticePropInF} and Remark \ref{rem:lls}).

Of course, one can find more continuity concepts in the literature. However, it is not the purpose of this paper to give a complete overwiew but to analyze the concepts used most frequently.
\end{remark}

Another useful concept for our considerations is that of efficiency of a set-valued function introduced by Verona and Verona \cite{VeronaVerona90} with $Z$ a Banach space. We generalize it here to arbitrary topological vector spaces.

\begin{definition}
A function $f:X\to \P(Z)$ is called {\em efficient} (eff.) at a point $x_0\in X$ iff there exist a neighborhood $U$ of $x_0$ in $X$ and a bounded set $B\subseteq Z$ such that $f(x)\cap B\neq \emptyset$ for all $x\in U$.
\end{definition}
Recall that a subset $B$ of a topological vector space is called bounded iff it is absorbed by every neighborhood of the origin (i.e., for every $V\in \V(0)$ exists some $r>0$ with $B\subseteq rV$).

The following implications are proven in \cite{GRTZ}:
\[
\begin{xy} \xymatrix{
\text{ eff. }\quad &
\quad\text{ l.c. }\quad \ar@<5pt>@2{<-}[l]^{\substack{f\text{ $C$--convex}\\ f(x_0)\subseteq D+C}} \ar@<5pt>@2{<-}[r] &
\quad\text{ H-l.c. }\quad \ar@<5pt>@2{<-}[l]^{\substack{f\text{ $C$--convex}\\ f(x_0)\subseteq D+C}\quad} &
\quad\text{ H-u.c. }\quad \ar@<5pt>@2{<-}[l]^{\substack{f\text{ $C$--convex}\\ f(x_0)\subseteq D+C}\;\,} & 
\quad\text{ u.c. }\quad \ar@<-5pt>@2{->}[l]
}\end{xy}
\]
Here all conditions are supposed to hold at $x_0\in \dom f$. Moreover, $f(x_0)\subseteq D+C$ means that there exists some bounded set $D\subseteq Z$ with this property.

\begin{remark}
In \cite{GRTZ} also the concepts of $C$-upper continuity, $C$-lower continuity, $C$-Hausdorff upper continuity and $C$-Hausdorff lower continuity are defined. The concepts of $C$-l.c., $C$-H-l.c. and $C$-H-u.c. coincide with l.c., H-l.c. and H-u.c., respectively, for functions mapping into $\F(Z,C)$. $C$-u.c. should be placed between H-u.c. and u.c. in the above diagram but we will omit it since it is not essential for our purpose of working out a suitable regularity condition.
\end{remark}

\begin{remark} In \cite[Theorem 2.6.6]{GRTZ} the fact that efficiency implies lower continuity was proven under the additional assumption that there is some bounded set $D$ with $f(x_0)\subseteq D+C$. This condition is not necessary, as the following proposition shows.
\end{remark}

\begin{proposition}
Let $f:X\to\F(Z,C)$ be convex. If $f$ is efficient at $x_0\in X$ then $f$ is lower continuous at $x_0$.
\end{proposition}

\begin{proof}
By efficiency of $f$ at $x_0$ there is some neighborhood $W$ of $0$ in $X$ and a bounded set $B\subseteq Z$ such that $f(x_0+w)\cap B \neq \emptyset$ for every $w\in W$. Let $z_0\in f(x_0)$ and $V\in\V(0)$. Then there is some balanced neighborhood $V'$ of $0$ with $V'\subseteq V$. Since $B$ is bounded, $B-\cb{z_0}$ is bounded as well, and there is some $t\in (0,1)$ such that $t(B-\cb{z_0})\subseteq V'\subseteq V$.

Let $U:=x_0+tW$. For all $x\in U$ there is some $w\in W$ with $x=x_0+tw=(1-t)x_0+t(x_0+w)$. Since $f(x_0+w)\cap B \neq \emptyset$ there is some $b\in f(x_0+w)\cap B$. By convexity of $f$ we obtain 
\[
f(x)\supseteq (1-t)f(x_0)+tf(x_0+w)\ni (1-t)z_0+tb=z_0+t(b-z_0)\in \cb{z_0}+V.
\]
Hence $f(x)\cap \of{\cb{z_0}+V} \neq \emptyset$.
\end{proof}

Moreover, we can show that under some additional condition lower continuity implies efficiency.

\begin{proposition}
Assume that the condition
\[
{\rm(BN)}\hspace{1cm} \text{there is a bounded set }B\subseteq Z \text{ and some }V\in\V(0)\text{ with } V\subseteq B-C 
\]
is satisfied.

If $f:X\to\F(Z,C)$ is lower continuous at $x_0\in \dom f$ then $f$ is efficient at $x_0$.
\end{proposition}

\begin{proof}
Since $x_0\in \dom f$ there is some $z_0\in f(x_0)$. Let $B$ be a bounded set in $Z$ and $V$ be a neighborhood of $0$ in $Z$ with $V\subseteq B-C$. By lower continuity there is some $U\in\NN(x_0)$ such that $f(x)\cap (\cb{z_0}+V)\neq \emptyset$ for every $x\in U$. Since $V\subseteq B-C$ we obtain $f(x)\cap (\cb{z_0}+B-C)\neq \emptyset$. Hence $(f(x)+C)\cap (\cb{z_0}+B)\neq \emptyset$ for every $x\in U$, implying efficiency at $x_0$ since $f(x)+C\subseteq f(x)$ and $\cb{z_0}+B$ is bounded.
\end{proof}

\begin{remark}
Each of the following two conditions is sufficient for (BN).

(i) $\Int C \neq \emptyset,\qquad$ (ii) $Z$ is locally bounded.
\end{remark}

Since the usual notions of semicontinuity and boundedness for extended real-valued functions can be expressed by infimum and supremum in the image space, they can be generalized to functions mapping into a complete lattice in the following way.

\begin{definition}\label{def:latticeprop}
Let $(Y,\le)$ be a complete lattice. We denote the top element of $(Y,\le)$ by $\infty$.

(i) $f:X\to Y$ is called {\em lattice-bounded above} (l-b.a.) on a set $M\subseteq X$ iff there is some $y\in Y\setminus\cb{\infty}$ such that $f(x)\le y$ for all $x\in M$.

(ii) %Let $X$ be a topological space and $\NN(x_0)$ the system of neighborhoods of $x_0$ in $X$. 
$f:X\to Y$ is called {\em lower  lattice-semicontinuous} (l.s.c.) at $x_0\in X$ iff
\[f(x_0)\le \sup_{U\in\NN(x_0)}\inf_{x\in U}f(x).\] 
 
(iii) %Let $X$ be a topological space and $\NN(x_0)$ the system of neighborhoods of $x_0$ in $X$. 
$f:X\to Y$ is called {\em upper lattice-semicontinuous} (u.s.c.) at $x_0\in X$ iff
\[f(x_0)\ge \inf_{U\in\NN(x_0)}\sup_{x\in U}f(x).\] 
\end{definition}

Next, we consider the special case where $(Y,\le)=(\F(Z,C),\supseteq)$.

\begin{remark}\label{rem:LatticePropInF}
If $(Y,\le)=(\F(Z,C),\supseteq)$ and $f:X\to\F(Z,C)$, then the above notions can be specified in the following way.

(i) $f$ is lattice-bounded above on some set $M\subseteq X$ iff there is some $a\in Z$ such that $a\in f(x)$ for all $x\in M$.

(ii) $f$ is lower lattice-semicontinuous at $x_0$ iff 
\[
f(x_0)\supseteq\bigcap_{U\in\NN(x_0)}\cl \bigcup_{x\in U}f(x),
\]
i.e.,
\begin{equation}\label{CharLSC}
\forall z_0\not\in f(x_0) \; \exists U\in\NN(x_0)\;\exists V\in\V(z_0) \; \forall x\in U\;\forall z\in V : z\not\in f(x).
\end{equation}

(iii) $f$ is upper lattice-semicontinuous at $x_0$ iff 
\[
f(x_0)\subseteq\cl\bigcup_{U\in\NN(x_0)} \bigcap_{x\in U}f(x),
\]
i.e.,
\begin{equation}\label{CharUSC}
\forall z_0\in f(x_0)\;\forall V\in\V(z_0) \; \exists U\in\NN(x_0)\;\exists z\in V \; \forall x\in U : z\in f(x).
\end{equation}
\end{remark}

\begin{remark}\label{rem:lls}
Note that lower lattice-semicontinuity for $\F(Z,C)$-valued functions coincides with other concepts for set-valued functions. In \cite{GRTZ} a function satisfying property \eqref{CharLSC} is called closed at $x_0$, and outer semicontinuity is a commonly used term for that property as well (see, e.g., \cite{Beer} and \cite{RoWe98}). See also \cite{HL2011} for properties of lower lattice-semicontinuous $\F(Z,C)$-valued functions. 

The above definitions of upper semicontinuity and boundedness seem to be non-standard for set-valued functions.

Moreover, several notions of semicontinuity for functions mapping into ordered topological spaces are introduced in the literature. We refer the interested reader to the papers by Penot and Thera \cite{PenotThera82} and Beer \cite{Beer87}. However, $\F(Z,C)$ is not equipped with a topology by nature and it seems much more sensible to adopt notions that rely on the complete lattice property that $\F(Z,C)$ naturally has rather than defining topologies on $\F(Z,C)$ that make those notions applicable. For this reason, we restrict our considerations to the above mentioned concepts.
\end{remark}

Subsequently, we analyze the relationships between these notions and compare them with the classic concepts.

\begin{proposition}
If $f:X\to\F(Z,C)$ is lattice-bounded above on some neighborhood of $x_0\in X$ then $f$ is efficient at $x_0$. If $\Int C\neq \emptyset$ then also the converse is true.
\end{proposition}

\begin{proof}
Obviously, lattice-boundedness from above on some neighborhood of $x_0$ implies efficiency at $x_0$. For the converse, assume that $U$ is a neighborhood of $x_0$ and $B\subseteq Z$ is a bounded set such that $f(x)\cap B\neq \emptyset$ for every $x\in U$. Since $\Int C\neq \emptyset$ there is some $k\in C$ and a neighborhood $V$ of $0$ in $Z$ with $k-V\subseteq C$. By the boundedness of $B$ there is some $t>0$ such that $B\subseteq tV$. Since $C$ is a cone we have $tk-tV\subseteq C$, i.e., $B\subseteq tV\subseteq tk-C$. Hence $f(x)\cap (tk-C)\neq \emptyset$, i.e., $tk\in f(x)+C\subseteq f(x)$ for every $x\in U$.
\end{proof}

\begin{proposition}
If $f:X\to\F(Z,C)$ is upper lattice-semicontinuous at $x_0\in \dom f$ then there exists a neighborhood $U$ of $x_0$ such that $f$ is lattice-bounded above on $U$. 
\end{proposition}

\begin{proof}
Since $x_0\in\dom f$, $f(x_0)$ is nonempty. Choose $z_0\in f(x_0)$ and $V\in\V(z_0)$ arbitrarily. By upper lattice-semicontinuity there is some $U\in \NN(x_0)$ and $z\in V$ such that $z\in f(x)$ for all $x\in U$. Hence f is lattice-bounded above on $U$.  
\end{proof}

\begin{proposition}
If $f:X\to\F(Z,C)$ is convex and lattice-bounded above on some neighborhood of $x_0\in X$, then $f$ is upper lattice-semicontinuous at $x_0$.
\end{proposition}

\begin{proof}
By the boundedness assumption there exist a balanced neighborhood $W$ of $0$ in $X$ and some $a\in Z$ with $a\in f(x_0+w)$ for all $w \in W$. We will show that \eqref{CharUSC} holds. 

Let $z_0\in f(x_0)$ and $V\in\V(z_0)$ be chosen arbitrarily. Then $V-\cb{z_0}\in\V(0)$, and there exists some $t\in (0,1)$ such that $t(a-z_0)\in V-\cb{z_0}$, i.e., $z:=(1-t)z_0+ta \in V$. Let $U:=\cb{x_0}+tW$. For every $x\in U$ there is some $w\in W$ with $x=x_0+tw=(1-t)x_0+t(x_0+w)$. From the convexity of $f$ we obtain 
\[
f(x)\supseteq (1-t)f(x_0)+tf(x_0+w)\ni (1-t)z_0+ta =z.
\]
Hence \eqref{CharUSC} is satisfied.
\end{proof}

%\begin{proposition}
%Let $X$ and $Z$ be locally convex topological vector spaces. If $f: X \to \F(Z,C)$ is convex and weakly upper bounded on a neighborhood of $x_0$ then $f$ is lower lattice-semicontinuous at $x_0$.
%\end{proposition}
%
%\begin{proof}
%By the boundedness assumption there is a balanced neighborhood $W$ of $0$ in $X$ and some bounded set $B\subseteq Z$ such that $f(x_0-w)\cap B\neq \emptyset$ for all $w\in W$. We want to show that \eqref{CharLSC} holds. Let $z_0\not\in f(x_0)$. Since $f(x_0)$ is closed, there is a neighborhood $V$ of $0$ in $Z$ such that $(z_0+V)\cap f(x_0)=\emptyset$. Moreover, there is some $V'\in\V(0)$ with $V'+V'\subseteq V$. Since $B$ is bounded, there is some $t>1$ such that $B-z_0\subseteq tV'$.
%We have
%\begin{equation}\label{Gl1}
%\begin{split}
%z_0+V\supseteq z_0+V'+V'\supseteq z_0+\frac{t-1}{t}V'+V'&=\frac{t-1}{t}(z_0+V')+\frac{1}{t}z_0+V'\\
%&\supseteq \frac{t-1}{t}(z_0+V')+\frac{1}{t}B.
%\end{split}
%\end{equation}
%Let $x=x_0+\frac{1}{t-1}w$ for some $w\in W$. Then 
%\[
%x_0=\frac{t-1}{t}x+\frac{1}{t}(x_0-w),
%\] 
%hence
%\[
%f(x_0)\supseteq \frac{t-1}{t}f(x)+\frac{1}{t}f(x_0-w)\supseteq \frac{t-1}{t}f(x)+\frac{1}{t}b_w
%\]
%with some $b_w\in f(x_0-w)\cap B$. 
%
%From \eqref{Gl1} we obtain 
%\[
%\frac{t-1}{t}z+\frac{1}{t}b_w\in z_0+V
%\]
%hence
%\[
%\frac{t-1}{t}z+\frac{1}{t}b_w\not\in f(x_0)\supseteq \frac{t-1}{t}f(x)+\frac{1}{t}b_w
%\]
%for all $z\in z_0+V'$.
%
%Consequently, $z\not\in f(x)$ for all $x\in x_0+ \frac{1}{t-1}W=:U\in \NN(x_0)$ and all $z\in z_0+V'\in\V(z_0)$, hence \eqref{CharLSC} is satisfied.
%\end{proof}
%
\begin{proposition}
If $f:X\to\F(Z,C)$ is convex and lower continuous at $x_0\in \dom f$ then $f$ is lower lattice-semicontinuous at $x_0$.
\end{proposition}

\begin{proof}
We want to show that \eqref{CharLSC} holds. Let $z_0\not\in f(x_0)$. Since $f(x_0)$ is closed there is a neighborhood $V$ of $0$ in $Z$ such that $(\cb{z_0}+V)\cap f(x_0)=\emptyset$. Moreover, there is some balanced neighborhood $V'$ of $0$ in $Z$ with $V'+V'+V'\subseteq V$. Since $x_0\in\dom f$ there is some $y_0\in f(x_0)$ and some $t\in(0,1)$ with $t(y_0-z_0)\in V'$. 

By lower continuity there is some neighborhood $W$ of $0$ in $X$ such that $f(x_0-w)\cap (\cb{y_0}+V')\neq \emptyset$ for every $w\in W$.
We have
\begin{equation}\label{Gl2}
\begin{split}
\cb{z_0}+V&\supseteq \cb{z_0}+V'+V'+V'\supseteq \cb{z_0}+V'+(1-t)V'+tV'\\
&\supseteq \cb{z_0}+\cb{t(y_0-z_0)}+(1-t)V'+tV'\\
&=(1-t)\of{\cb{z_0}+V'}+t\of{\cb{y_0}+V'}.
\end{split}
\end{equation}
Let $x=x_0+\frac{t}{1-t}w$ for some $w\in W$, i.e., $x_0=(1-t)x+t(x_0-w)$, and let $z\in\cb{z_0}+V'$. We want to show that $z\not\in f(x)$. 
Assuming on the contrary that $z\in f(x)$, we obtain
\[
f(x_0)\supseteq (1-t)f(x)+tf(x_0-w)\ni (1-t)z+ty
\]
with some $y\in f(x_0-w)\cap (y_0+V')$. 
From \eqref{Gl2} we obtain $(1-t)z+ty\in \cb{z_0}+V$, contradicting $(\cb{z_0}+V)\cap f(x_0)=\emptyset$.
Consequently, $z\not\in f(x)$ for all $x\in \cb{x_0}+ \frac{t}{1-t}W=:U\in \NN(x_0)$ and all $z\in \cb{z_0}+V'\in\V(z_0)$. Hence \eqref{CharLSC} is satisfied.
\end{proof}

%Obviously, lower lattice-semicontinuity does not imply upper boundedness as the following simple example shows.

\begin{proposition}
If $f:X\to\F(Z,C)$ is upper lattice-semicontinuous at $x_0\in X$ then $f$ is lower continuous at $x_0$ as well. 

The converse statement is true if $\Int C \neq\emptyset$.
\end{proposition}

\begin{proof}
(i) By upper lattice-semicontinuity, \eqref{CharUSC} holds.  Let $z_0\in f(x_0)$ and $V\in\V(z_0)$.
By \eqref{CharUSC} there is some $U\in \NN(x_0)$ and $z\in V$ such that $z\in f(x)$ for all $x\in U$. Hence $f(x)\cap V \neq\emptyset$ for all $x\in U$.

(ii) Let $z_0\in f(x_0)$, $V\in \V(0)$ be chosen arbitrarily and take $k\in\Int C$. Then there is some $t>0$ such that $tk\in V$. Since $k\in\Int C$ there is some neighborhood $W$ of $0$ in $Z$ with $\cb{k}+W\subseteq C$. Since $C$ is a cone we have $t\cb{k}+tW\in C$, and $-tW$ is a neighborhood of $0$ in $Z$ as well. Since $f$ is lower continuous at $x_0$, there is some $U\in\NN(x_0)$ such that $f(x)\cap (\cb{z_0}-tW)\neq \emptyset$ for every $x\in U$, i.e., $z_0\in f(x)+tW$ for every $x\in U$. 
Hence $z_0+tk\in f(x)+tW+t\cb{k}\subseteq f(x)+C\subseteq f(x)$ for every $x\in U$, which proves upper lattice-semicontinuity. 
\end{proof}

\begin{proposition}
If $f:X\to\F(Z,C)$ is Hausdorff upper continuous at $x_0\in X$ then $f$ is lower lattice-semicontinuous at $x_0$.
\end{proposition}

\begin{proof}
Let $z_0\not\in f(x_0)$. Since $f(x_0)$ is closed there exists a neighborhood $W$ of $0$ in $Z$ such that $(z_0+W)\cap (f(x_0)+W)=\emptyset$. By Hausdorff upper continuity there exists some $U\in \NN(x_0)$ such that $f(x)\subseteq f(x_0)+W$ for all $x\in U$. Hence $z\not \in f(x)$ for all $x\in U$ and all $z\in \cb{z_0}+W$. Thus $f$ is lower lattice-semicontinuous at $x_0$.
\end{proof}

\begin{proposition}
If $f:X\to\F(Z,C)$ and there is some $z_0\in Z$ such that $(x_0,z_0)\in \Int(\gr f)$ then $f$ is lattice-bounded above on some neighborhood of $x_0$.
If $\Int C\neq \emptyset$ then the converse is true as well. 
\end{proposition}

\begin{proof}
If $(x_0,z_0)\in \Int(\gr f)$, then there are neighborhoods $U\in\NN(x_0)$ and $V\in\V(z_0)$ such that $z\in f(x)$ for all $x\in U$ and all $z\in V$. In particular, $z_0\in f(x)$ for all $x\in U$.

Now we will prove the converse. Since $\Int C\neq \emptyset$ there is some $k\in C$ and a neighborhood $V$ of $0$ in $Z$ with $\cb{k}+V\subseteq C$. If there is some neighborhood $U$ of $x_0$ and some $a\in Z$ with $a\in f(x)$ for all $x\in U$ then $\cb{a+k}+V\subseteq f(x)+C\subseteq f(x)$ for all $x\in U$. Hence $(x_0,a+k)\in\Int(\gr f)$. 
\end{proof}

The above statements can be summarized in the following diagram for $f:X\to\F(Z,C)$. Again, all properties are supposed to be valid locally at one point $x_0\in\dom f$. In this context, $\Int (\gr f)\neq \emptyset$ should be understood as ''there is some $z_0\in Z$ with $(x_0,z_0)\in \Int(\gr f)$''.
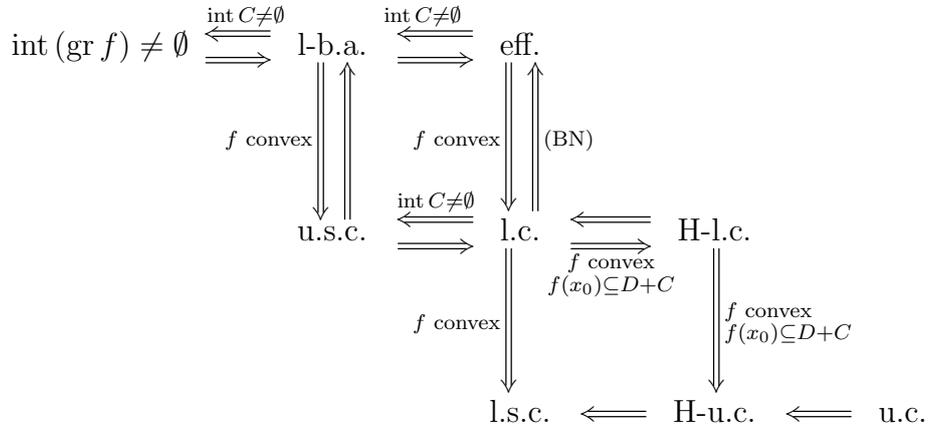
\[
\begin{xy} \xymatrix{
\Int (\gr f)\neq \emptyset\; \ar@<-5pt>@2{->}[r] & 
\;\text{ l-b.a. }\; \ar@<-5pt>@2{->}[l]_/-10pt/{\Int C\neq \emptyset} \ar@<-5pt>@2{->}[dd]_{f\text{ convex}} \ar@<-5pt>@2{->}[r] &
\;\text{ eff. }\; \ar@<-5pt>@2{->}[l]_/2pt/{\Int C\neq \emptyset} \ar@<5pt>@2{<-}[dd]^{\text{(BN)}} &
&
\\
\\
&
\; \text{ u.s.c. }\; \ar@<-5pt>@2{->}[r] \ar@2{->}@<-5pt>[uu] &
\;\text{ l.c. }\; \ar@<-5pt>@2{->}[l]_/-3pt/{\Int C\neq \emptyset} \ar@<-5pt>@2{->}[dd]_{f\text{ convex}} \ar@<5pt>@2{<-}[r] \ar@2{<-}@<5pt>[uu]^{f\text{ convex}} &
\;\text{ H-l.c. }\; \ar@2{->}[dd]^{\begin{subarray}{l}f\text{ convex}\\ f(x_0)\subseteq D+C\end{subarray}} \ar@<5pt>@2{<-}[l]^/3pt/{\substack{f\text{ convex}\\ f(x_0)\subseteq D+C}} \\
\\
& 
&
\;\text{ l.s.c. }\;\ar@2{<-}[r] &
\;\text{ H-u.c. }\; &
\;\text{ u.c. }\; \ar@2{->}[l] 
}\end{xy}
\]

\begin{remark}\label{rem:contvssemicont}
As one can see in the above diagram, (Hausdorff) lower continuity is closely related to upper lattice-semicontinuity.
This is based on the fact that in the classic concepts ''lower'' is related to subsets, but subsets are greater elements in the lattice $(\F(Z,C),\supseteq)$.
\end{remark}

The following examples provide counterexamples for most of the missing implications.

\begin{example}
Let $X=\R$, $Z=\R^2$ and $C=\cb{z\in \R^2 \st z_1=0, z_2\ge 0}$. For all $x_0\in \R$, the function $f:X\to \F(Z,C)$ defined by 
$f(x)=\cb{(x,0)^T}+C$ is convex and Hausdorff lower continuous but not upper lattice-semicontinuous at $x_0$.  
\end{example}

\begin{example}
Let $X=\R$ and define
\[
f(x)=\begin{cases} C & \text{ if }x\ge 0 \\ \emptyset & \text{ if }x<0. \end{cases}
\]
$f:X\to\F(Z,C)$ is convex and upper continuous at $x_0=0$, but $f$ is not efficient at $x_0=0$.
\end{example}

\begin{example}\label{ex:setA}
Let $A:=\cb{z\in\R^2 : z_2\ge z_1^2}$ and $f:\R\to \F(\R^2,\R^2_+)$ be defined by
\[
f(x)=\begin{cases} xA+\R^2_+ & \text{ if }x\ge 0 \\ \emptyset & \text{ if }x<0. \end{cases}
\] 
The function $f$ is convex, upper and lower lattice-semicontinuous at $x_0=1$ but neither Hausdorff upper continuous nor Hausdorff lower continuous at $x_0=1$ as the following lemma shows. 
\end{example}

\begin{lemma}
The set $A$ from Example \ref{ex:setA} has the following property 
\begin{equation}
\forall \varepsilon >0 :  (1+\varepsilon )A+\R^2_+\not\subseteq A+\R^2_+ + B_1(0).
\end{equation}
\end{lemma}
\begin{proof}
Take an arbitrary $\varepsilon >0$. We will show that there exists some $t>0$ such that for the point $\bar z=(-t,t^2)^T\in A$ the distance of $(1+\varepsilon)\bar z$ to $A+\R^2_+$ is greater than $1$. In fact, the distance of $(1+\varepsilon)\bar z$ to the tangent to the graph of the function $z_2=z_1^2$ through the point $\bar z$ (which is smaller than the distance to $A+\R^2_+$) equals
\[\frac{\varepsilon t^2}{\sqrt{1+4t^2}},\]
tending to $\infty$ if $t$ tends to $\infty$.
\end{proof}

As in the scalar case, we can show that for convex functions boundedness from above and upper semicontinuity carries over from one point of the domain to any other point in the interior of the domain.

\begin{lemma}
Let $f:X\to\F(Z,C)$ be convex and $x_0\in \dom f$. If $f$ is lattice-bounded above on a neighborhood of $x_0$ or efficient at $x_0$ or upper lattice-semicontinuous at $x_0$ or lower continuous at $x_0$ then $f$ has the corresponding property at all $x\in\Int(\dom f)$.
\end{lemma}

\begin{proof}
For lower continuity the statement is proven in \cite[Theorem 2.6.1]{GRTZ}. We will now prove it for efficiency.

Let $f$ be efficient at $x_0$, i.e., there is a bounded set $B$ in $Z$ such that $f(u)\cap B \neq  \emptyset$ for all $u\in U$. If $x\in\Int\of{\dom f}$ then there is some $t>0$ such that
$y:=x+t(x-x_0)\in \dom f$. Let $z\in f(y)$ and
\[
W:=\cb{x}+\frac{t}{1+t}\of{U-\cb{x_0}}=\frac{t}{1+t}U+\cb{x-\frac{t}{1+t}x_0}=\frac{t}{1+t}U+\cb{\frac{1}{1+t}y}.
\] 
Then $W$ is a neighborhood of $x$. For every $w\in W$ there is some $u\in U$ with $w= \frac{t}{1+t}u+\frac{1}{1+t}y$. By convexity of $f$ we get
\[
f(w)\supseteq \frac{t}{1+t}f(u)+\frac{1}{1+t}f(y)\supseteq \frac{t}{1+t}f(u)+\cb{\frac{1}{1+t}z}.
\]
The set $\widetilde{B}:=\frac{t}{1+t}B+\cb{\frac{1}{1+t}z}$ is a bounded set and $\frac{t}{1+t}(f(u)\cap B) + \cb{\frac{1}{1+t}z} \subseteq f(w)\cap\widetilde{B}$ implies $f(w)\cap\widetilde{B}\neq \emptyset$. Hence $f$ is efficient at $x_0$. 

The case of lattice boundedness can be treated in the same way by replacing $B$ by a singleton $\cb{a}$ in the considerations above. By equivalence between lattice-boundedness from above and upper lattice-semicontinuity, the statement for upper lattice-semicontinuity is proven as well.
\end{proof}

In the Fenchel-Rockafellar type duality theorems for set-valued optimization mentioned in the introduction, mainly two methods of proof have been used. In \cite{Hamel11} Hamel directly applies a separation theorem in $X\times Z$ under the assumption that $\Int(\gr g)\neq \emptyset$, whereas Schrage \cite{DissSchrage}, \cite{Schrage11} and L\"ohne \cite{BookLoehne} assume that all scalarizations are continuous (in fact L\"ohne defines a topology on $\F(Z,C)$ in such a way that a function $f:X\to\F(Z,C)$ is continuous with respect to this topology if and only if all scalarizations are continuous) and apply the scalar Fenchel-Rockafellar theorem to the scalarizations. For this reason we next analyze the relationships between the semicontinuity concepts for a set-valued function and semicontinuity of its scalarizations.

\begin{proposition}
If $f:X\to\F(Z,C)$ is lower continuous at $x_0\in X$ then $\varphi_{(f,z^*)}$ is upper semicontinuous at $x_0$ for every $z^*\in Z^*$.
\end{proposition}

\begin{proof}
Let $z^*\in Z^*$. We distinguish 3 cases.

1. $\varphi_{(f,z^*)}(x_0)=+\infty$. Then $\varphi_{(f,z^*)}$ is obviously upper semicontinuous at $x_0$.

2. $\varphi_{(f,z^*)}(x_0)\in \R$.
Let $\varepsilon >0$. Then there exists some $z_\varepsilon\in f(x_0)$ such that 
\[
-z^*(z_\varepsilon) < \inf_{z\in f(x_0)}(-z^*(z)) + \varepsilon = \varphi_{(f,z^*)}(x_0)+\varepsilon.
\]
$V:=\cb{z\in Z \st -z^*(z) < \varphi_{(f,z^*)}(x_0)+\varepsilon}$ is a neighborhood of $z_\varepsilon$ in $Z$. From the lower continuity of $f$ at $x_0$ we obtain the existence of some neighborhood $U$ of $x_0$ in $X$ with $f(x)\cap V\neq \emptyset$ for every $x\in U$. Hence $\varphi_{(f,z^*)}(x)< \varphi_{(f,z^*)}(x_0)+\varepsilon$ for every $x\in U$, implying upper semicontinuity of $\varphi_{(f,z^*)}$ at $x_0$.

3. $\varphi_{(f,z^*)}(x_0)=-\infty$.
Let $n\in \N$. Then there exists some $z_n\in f(x_0)$ such that $-z^*(z_n) < -n$.
$V:=\cb{z\in Z \st -z^*(z) < -n}$ is a neighborhood of $z_n$ in $Z$. From the lower continuity of $f$ at $x_0$ we obtain the existence of some neighborhood $U$ of $x_0$ in $X$ with $f(x)\cap V\neq \emptyset$ for every $x\in U$. Hence $\varphi_{(f,z^*)}(x)< -n$ for every $x\in U$, implying upper semicontinuity of $\varphi_{(f,z^*)}$ at $x_0$.
\end{proof}
The following example shows that the converse is not true in general.

\begin{example}
Let $X=\R$, $Z=\R^2$, $C=\R^2_+$ and 
\[
f(x)=
  \begin{cases}
  \cb{z\in\R^2\st z_1+xz_2\ge 1+x} & \text{ if } x>0 \\
  \R^2_+ & \text{ if } x\le 0.
  \end{cases}
\]
If $z_0 =(0,0)^T\in f(0)$ and $V=\cb{z\in\R^2 \st \abs{z_1}+\abs{z_2}< 1}\in \V(z_0)$ then
\[
\abs{z_1}+\abs{z_2}\ge z_1+xz_2\ge 1+x > 1
\]
for all $x\in (0,1)$ and $z\in f(x)$, i.e., $f(x)\cap V =\emptyset$. Hence $f$ is not lower continuous at $x_0=0$.

For the scalarizations we have
\[
  \varphi_{(f,z^*)}(x)=
  \begin{cases}
		-(z_1^*+z_2^*) & \text{ if } x>0 \text{ and } xz_1^*=z_2^* \\
		-\infty            & \text{ if } x>0 \text{ and } xz_1^*\neq z_2^* \\
		0                  & \text{ if } x\le 0
  \end{cases}
\]
if $z^*\in-\R^2_+$ and $\varphi_{(f,z^*)}\equiv -\infty$ if $z^*\not\in-\R^2_+$.
Hence the scalarizations $\varphi_{(f,z^*)}$ are upper semicontinuous at $x_0=0$ for all $z^*\in \R^2$.
\end{example}

\begin{proposition}
If $f:X\to\F(Z,C)$ is Hausdorff upper continuous at $x_0\in X$ then $\varphi_{(f,z^*)}$ is lower semicontinuous at $x_0$ for every $z^*\in Z^*$.
\end{proposition}

\begin{proof}
Let $z^*\in Z^*$ and $\varepsilon >0$. Then $V:=\cb{z\in Z \st -z^*(z) > - \varepsilon}$ is a neighborhood of $0$ in $Z$. From the Hausdorff upper continuity of $f$ at $x_0$ we obtain the existence of some neighborhood $U$ of $x_0$ in $X$ with $f(x)\subseteq f(x_0)+V$ for every $x\in U$. Hence $\varphi_{(f,z^*)}(x)\ge \varphi_{(f,z^*)}(x_0)-\varepsilon$ for every $x\in U$, implying lower semicontinuity of $\varphi_{(f,z^*)}$ at $x_0$.
\end{proof}

It is easy to show that in Example \ref{ex:setA} all scalarizations of $f$ are continuous at $x_0=1$. Hence the converse of the preceding proposition is not true in general.

\begin{proposition}
Let $Z$ be a locally convex space and $f:X\to \F(Z,C)$. If $f(x_0)$ is a convex set and $\varphi_{(f,z^*)}$ is lower semicontinuous at $x_0$ for every $z^*\in C^-\setminus\cb{0}$ then $f$ is lower lattice-semicontinuous at $x_0$.
\end{proposition}

\begin{proof}
Assume that $f$ is not lower lattice-semicontinuous at $x_0$. Then
\[
\exists z_0\not\in f(x_0) \; \forall U\in\NN(x_0)\; \forall V\in\V(z_0) \; \exists x\in U : V\cap f(x)\neq\emptyset.
\]
Since $f(x_0)$ is closed and convex, we can separate $z_0$ and $f(x_0)$ strictly. Hence there exist $z^*\in C^-\setminus\cb{0}$ and $\alpha \in \R$ with
\begin{equation}\label{notlsc}
-z^*(z_0)< \alpha <\inf_{z\in f(x_0)} (-z^*(z))=\varphi_{(f,z^*)}(x_0).
\end{equation}
Let $V:=\cb{z\in Z \st -z^*(z)<\alpha}\in \V(z_0)$. Then for every $U\in \NN(x_0)$ there exists some $x\in U$ with $V\cap f(x)\neq\emptyset$ due to \eqref{notlsc}. Hence 
\[\inf_{x\in U}\varphi_{(f,z^*)}(x)< \alpha < \varphi_{(f,z^*)}(x_0)
\] 
for every $U\in \NN(x_0)$, contradicting the lower semicontinuity of $\varphi_{(f,z^*)}$ at $x_0$.
\end{proof}

Again, Example \ref{ex:setA} can be taken as a counterexample for the converse statement. The function $f$ in Example \ref{ex:setA} is lower lattice-semicontinuous at $x_0=0$, but for $-z^*=(1,0)^T$ the scalarization is not lower semicontinuous at $0$.

Semicontinuity of all scalarizations $\varphi_{(f,z^*)}$ with $z^*\in C^-\setminus\cb{0}$, denoted by $C^-$-l.s.c. and $C^-$-u.s.c., respectively, can be incorporated into our diagram as follows.

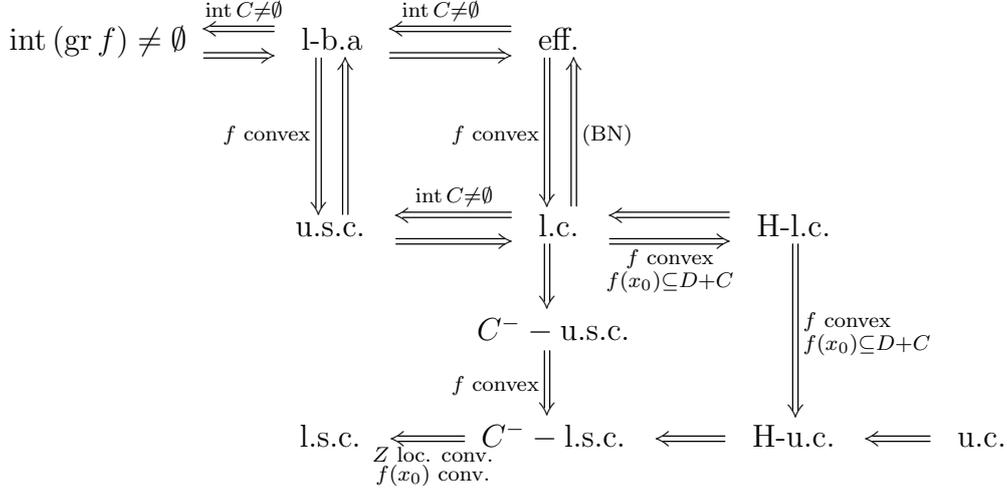
\[
\begin{xy} \xymatrix{
\Int (\gr f)\neq \emptyset\; \ar@<-5pt>@2{->}[r] & 
\;\text{ l-b.a }\; \ar@<-5pt>@2{->}[l]_/-10pt/{\Int C\neq \emptyset} \ar@<-5pt>@2{->}[dd]_{f\text{ convex}} \ar@<-5pt>@2{->}[r] &
\;\text{ eff.
 }\; \ar@<-5pt>@2{->}[l]_/2pt/{\Int C\neq \emptyset} \ar@<5pt>@2{<-}[dd]^{\text{(BN)}} &
&
\\
\\
& 
\; \text{ u.s.c. }\; \ar@<-5pt>@2{->}[r] \ar@2{->}@<-5pt>[uu] &
\;\text{ l.c. }\; \ar@<-5pt>@2{->}[l]_/-3pt/{\Int C\neq \emptyset} \ar@<5pt>@2{<-}[r] \ar@<-5pt>@2{->}[d] \ar@2{<-}@<5pt>[uu]^{f\text{ convex}} &
\;\text{ H-l.c. }\; \ar@2{->}[dd]^{\begin{subarray}{l}f\text{ convex}\\ f(x_0)\subseteq D+C\end{subarray}} \ar@<5pt>@2{<-}[l]^/3pt/{\substack{f\text{ convex}\\ f(x_0)\subseteq D+C}} \\
& &
\; C^--\text{u.s.c. }\; \ar@<-5pt>@2{->}[d]_{f\text{ convex}} &
\\
& 
\; \text{ l.s.c. }\; \ar@2{<-}[r]_/-5pt/{\substack{Z\text{ loc. conv.}\\ f(x_0)\text{ conv.}}} &
\; C^--\text{l.s.c. }\; \ar@2{<-}[r] &
\;\text{ H-u.c. }\; &
\;\text{ u.c. } \ar@2{->}[l] 
}\end{xy}
\]

We have seen by counterexamples that in general none of the lower continuity (or upper semicontinuity) concepts will be implied by $C^--$ upper semicontinuity. In order to guarantee lower continuity, one needs some kind of uniform upper semicontinuity of the scalarizations.

\begin{theorem}
Let $Z$ be a locally convex space and $f:X\to\F(Z,C)$ be convex-valued. Assume that there is a set $B\subseteq  Z^*$ with $\cone B = C^-$ and 
\begin{equation}\label{infsup}
\forall z\in Z : \sup_{z^*\in B}z^*(z)<\infty \quad \text{ and }\quad     \forall V\in \V(0) : \inf_{z^*\in B}\sup_{z\in V}\sqb{-z^*(z)}>0
\end{equation}
such that the scalarizations $\varphi_{(f,z^*)}$ are upper semicontinuous at $x_0$ uniformly with respect to $B$, i.e., 
\begin{multline}\label{unifusc}
\forall \varepsilon >0 \;\exists U\in \NN(x_0) \; \forall x\in U\; \forall z^*\in B : \\
\begin{cases}
  \varphi_{(f,z^*)}(x)<-\frac{1}{\varepsilon} & \text{ if }\varphi_{(f,z^*)}(x_0)=-\infty \\
  \varphi_{(f,z^*)}(x)<\varphi_{(f,z^*)}(x_0)+\varepsilon & \text{ otherwise.}
\end{cases}
\end{multline}
Then $f$ is lower continuous at $x_0$.
\end{theorem}

\begin{proof}
Assume that $f$ is not lower continuous at $x_0$. Then there exist $z_0\in f(x_0)$ and $V\in \V(0)$ such that for all $U\in \NN(x_0)$ there is some $x_U\in U$ with $f(x_U)\cap (z_0+V) = \emptyset$. Since $Z$ is locally convex, $V$ can be assumed to be convex. Weakly separating $f(x_U)$ and $\cb{z_0}+V$, we get
\[
\forall U\in\NN(x_0)\;\exists x_U\in U\; \exists z^*_U\in B : -z^*_U(z_0)+\sup_{z\in V}\sqb{-z_U^*(z)}\le \varphi_{(f,z^*_U)}(x_U).
\]
Choose $\varepsilon >0$ such that
\[
\varepsilon <\inf_{z^*\in B}\sup_{z\in V}\sqb{-z^*(z)}
\quad\text{ and }\quad \frac{1}{\varepsilon} >\sup_{z^*\in B}z^*(z_0).
\]
By assumption \eqref{infsup} such $\varepsilon$ always exists.
By \eqref{unifusc} there is a neighborhood $\bar U \in \NN(x_0)$ such that $\varphi_{(f,z^*)}(x)<\varphi_{(f,z^*)}(x_0)+\varepsilon$ for all $x\in \bar U$ and $z^*\in B$ with $\varphi_{(f,z^*)}(x_0)>-\infty$ and $\varphi_{(f,z^*)}(x)<-\frac{1}{\varepsilon}$ if $\varphi_{(f,z^*)}(x_0)=-\infty$. In particular, if $\varphi_{(f,z^*_{\bar U})}(x_0)>-\infty$ then
\begin{align*}
\varphi_{(f,z^*_{\bar U})}(x_{\bar U})<\varphi_{(f,z^*_{\bar U})}(x_0)+\varepsilon \le -z^*_{\bar U}(z_0)+\sup_{z\in V}\sqb{-z_{\bar U}^*(z)}\le \varphi_{(f,z^*_{\bar U})}(x_{\bar U}),
\end{align*}
a contradiction, and if $\varphi_{(f,z^*_{\bar U})}(x_0)=-\infty$, then
\begin{align*}
\varphi_{(f,z^*_{\bar U})}(x_{\bar U})<-\frac{1}{\varepsilon} \le -z^*_{\bar U}(z_0)+\sup_{z\in V}\sqb{-z_{\bar U}^*(z)}\le \varphi_{(f,z^*_{\bar U})}(x_{\bar U}),
\end{align*}
a contradiction, too.
\end{proof}
The assumption of the existence of a set $B$ in the preceding theorem is not very restrictive as the following remark shows.
\begin{remark}
(i)%{\color{red} Do we need Hausdorff lcs here??} 
If $B$ is weak$^*$-compact and $0\not\in B$ then \eqref{infsup} is satisfied. This follows directly from the fact that a continuous real-valued function has a finite supremum over a compact set and that a lower semicontinuous function attains its infimum over a compact set, taking into account that $\sup_{z\in V}\sqb{-z^*(z)}>0$ if $z^*\neq 0$.

(ii) If $Z$ is a Hausdorff locally convex space and $\Int C\neq\emptyset$, then $C^-$ has a weak$^*$-compact base $B$ (see e.g.%{\color{red} Goepfert et al wrongly cite Prop. 5 of a paper by Zalinescu for the proof, Zalinescu cites Aubin for the proof (p.22, Remark 11)}
\cite[Lemma 2.2.17]{GRTZ}). Since a base of the cone $C^-$ is defined to be a convex set $B$ with $0\not\in \cl B$ and $\cone B =C^-$, there is a set $B\subseteq Z^*$ satisfying $\cone B = C^-$ and \eqref{infsup}.

(iii) If $Z$ is a normed space then $B:=\cb{z^*\in C^- \st \norm{z^*}_*=1}$ satisfies \eqref{infsup} and $\cone B=C^-$.
Indeed,
\[
\sup_{z^*\in B}z^*(z)\le \sup_{z^*\in B}\norm{z^*}_*\norm{z}\le\norm{z}<\infty
\]
and for every $V\in\V(0)$ there is some $\delta >0$ with $V\supseteq\cb{z\in Z \st \norm{z}\le\delta}$, hence for $z^*\in B$ we obtain
\[
\sup_{z\in V} \sqb{-z^*(z)}\ge\sup_{\norm{z}\le\delta}z^*(-z)=\sup_{\norm{z'}\le 1}z^*(\delta z')=\delta \norm{z^*}_*=\delta.
\]

\end{remark}

Analogously, it can be shown that uniform lower semicontinuity of the scalarizations implies Hausdorff upper continuity.

\begin{theorem}
Let $Z$ be a locally convex space and $f:X\to\F(Z,C)$ be convex-valued. Assume that there is a set $B\subseteq  Z^*$ with $\cone B = C^-$ and \eqref{infsup} 
%\begin{equation}\label{supinf}
%\forall V\in \V(0) : \sup_{z^*\in B}\inf_{z\in V}\sqb{-z^*(z)}<0
%\end{equation}
such that the scalarizations $\varphi_{(f,z^*)}$ are lower semicontinuous at $x_0$ uniformly with respect to $B$, i.e., 
\begin{equation}\label{uniflsc}
\forall \varepsilon >0 \;\exists U\in \NN(x_0) \; \forall x\in U \; \forall z^*\in B : 
  \varphi_{(f,z^*)}(x)>\varphi_{(f,z^*)}(x_0)-\varepsilon .
\end{equation}
Then $f$ is Hausdorff upper continuous at $x_0$.
\end{theorem}

\begin{proof}
Assume that $f$ is not  Hausdorff upper continuous at $x_0$. Then there exists some neighborhood $V\in \V(0)$ such that for all $U\in \NN(x_0)$ there is some $x_U\in U$ with $f(x_U)\not\subseteq f(x_0)-V$. Hence there is some $z_U\in f(x_U)$ with $z_U\not \in f(x_0)-V$. Since $Z$ is locally convex, $V$ can be assumed to be convex. Weakly separating $z_U$ and $f(x_0)-V$, we get
\begin{multline*}
\forall U\in\NN(x_0)\;\exists x_U\in U\; \exists z_U\in f(x_U)\; \exists z^*_U\in B :\\[1em]
 -z^*_U(z_U)\le \varphi_{(f,z^*_U)}(x_0)+\inf_{z\in -V}\sqb{-z_U^*(z)}.
\end{multline*}
Choose 
\[
\varepsilon =-\sup_{z^*\in B}\inf_{z\in -V}\sqb{-z^*(z)}=\inf_{z^*\in B}\sup_{z\in V}\sqb{-z^*(z)}.
\]
By assumption \eqref{infsup} we have $\varepsilon > 0$.
By \eqref{uniflsc} there is a neighborhood $\bar U \in \NN(x_0)$ such that $\varphi_{(f,z^*)}(x)>\varphi_{(f,z^*)}(x_0)-\varepsilon$ for all $x\in \bar U$ and $z^*\in B$. In particular,
\begin{align*}
\varphi_{(f,z^*_{\bar U})}(x_0)-\varepsilon &< \varphi_{(f,z^*_{\bar U})}(x_{\bar U}) \le -z^*_{\bar U}(z_{\bar U})\\
&\le \varphi_{(f,z^*_{\bar U})}(x_0)+\inf_{z\in -V}\sqb{-z_{\bar U}^*(z)}\le \varphi_{(f,z^*_{\bar U})}(x_0)-\varepsilon,
\end{align*}
a contradiction.
\end{proof}

%%%%%%%%%%%%%%%%%%%%%%%%%%%%%%%%%%%%%%%%%%%%%%%%%%%%%%%%%%%%%%%%%%%%%%%%%%%%%%%%%%%%
\section{A fundamental duality formula for set-valued functions}\label{sec:FDF}

In this section we will prove a set-valued analogon of the following fundamental duality theorem. Throughout this section, $X$, $Y$ and $Z$ are locally convex Hausdorff spaces with topological duals $X^*$, $Y^*$ and $Z^*$, respectively and $C\subseteq Z$ is a closed convex cone with $0\in C\neq Z$. 

\begin{theorem}[\cite{Zalinescu}, Theorem 2.7.1]\label{ThmScalarFD}
Let $\Phi:X\times Y\to  \OLR$ be a proper convex function and $h:Y\to \OLR$, $h(y):=\inf_{x\in X}\Phi(x,y)$ the associated marginal function, and assume that there is some $x_0 \in X$ such that $(x_0,0)\in \dom \Phi$ and $\Phi(x_0,\cdot)$ is continuous at $0$.

Then 
\begin{equation}\label{eq:max}
h(0)=\inf_{x\in X} \Phi(x,0)=\max_{y^*\in Y^*}-\Phi^*(0,y^*).
\end{equation}
\end{theorem}

This theorem is a rather general result from which most of the known duality results (e.g. the Fenchel-Rockafellar duality theorem) can be derived by choosing a suitable function $\Phi$.

Next we consider the set-valued case.
For a function $f:X\times Y \to \F(Z,C)$ the marginal function $f_X:Y\to \F(Z,C)$, which replaces the function $h$ of the preceding theorem, is defined by 
\[
f_X(y):=\cl\bigcup_{x\in X}f(x,y).
\]

For the function $f_X$, the following properties can easily be shown.
\begin{lemma}\label{LemInf}
(i) If $f:X\times Y \to \F(Z,C)$ is convex then $f_X$ is convex as well. In particular, $f_X$ has convex values. 

(ii) For all $z^*\in C^-\setminus\cb{0}$ and all $y\in Y$,
\[
\varphi_{\of{f_X,z^*}}(y) =\inf_{x\in X}\varphi_{\of{f,z^*}}(y)
\]
holds true.
\end{lemma}

As a first step towards the set-valued version of the fundamental duality theorem, we will prove weak duality.

\begin{lemma}\label{lem:1}
  Let $f:X\times Y\to\F(Z,C)$. Then 
  \[
  f_X(0)\subseteq(-f^*)((0,y^*),z^*)
  \]
is satisfied for all $(y^*,z^*)\in Y^*\times C^-\setminus\cb{0}$.
\end{lemma}
\begin{proof}
By definition,
\begin{align*}
  (-f^*)((0,y^*),z^*)&= \cl\bigcup\limits_{(x,y)\in X\times Y} \of{f(x,y)+S_{((0,y^*),z^*)}(-x,-y)}\\
  &= \cl\bigcup\limits_{y\in Y}\of{\bigcup\limits_{x\in X} f(x,y) + S_{(y^*,z^*)}(-y) }\\
                &\supseteq \cl\of{\bigcup\limits_{x\in X} f(x,0) + S_{(y^*,z^*)}(0) }\\
        	  & \supseteq  f_X(0)
\end{align*}
holds for all $(y^*,z^*)\in Y^*\times C^-\setminus\cb{0}$.
\end{proof}

Note that the results of Lemma \ref{LemInf} and Lemma \ref{lem:1} are obviously also true under the weaker assumttion that $X$, $Y$ and $Z$ are merely topological vector spaces.

Next, we state and prove a set-valued version of the fundamental duality theorem.
As regularity condition we use upper semicontinuity (being equivalent to continuity) of all scalarizations that turned out to be the weakest of all considered upper semicontinuity properties in the general case.

\begin{theorem}\label{thm:FundDualThm}%\cite[theorem 2.7.1(iii)]{Zalinescu02}
  Let $f:X\times Y\to\F(Z,C)$ a convex function with $(x_0,0)\in\dom f$ for some $x_0\in X$. If $\varphi_{(f(x_0,\cdot),z^*)}:Y\to\OLR$ is u.s.c. in $0\in Y$ for all $z^*\in C^-\setminus\cb{0}$ then
  \begin{align*}
    f_X(0) & =\bigcap\limits_{(y^*,z^*)\in Y^*\times C^-\setminus\cb{0}} (-f^*)((0,y^*),z^*)%;\\
%    \exists y^*_0\in Y^*:\quad \cl\bigcup\limits_{x\in X} h(x,0) & =(-h^*)(0,y^*,z^*)
  \end{align*}
  and there exists a family $\cb{y^*_{z^*}\st z^*\in C^-\setminus\cb{0},\; \varphi_{(f,z^*)} \text{ is proper}}\subseteq Y^*$ such that
  \begin{align*}
  	f_X(0) & =\bigcap\limits_{\substack{z^*\in C^-\setminus\cb{0}\\ \varphi_{(f,z^*)} \text{ is proper}} }(-f^*)((0,y^*_{z^*}),z^*).
  \end{align*}
%  Moreover, for $\bar x\in X$ and $z^*\in C^-\setminus\cb{0}$ it  holds
%  \begin{align*}
%   & \exists y^*_{z^*}\in Y^*:\quad \cl\of{ f(\bar x,0)+S_{((0,y^*_{z^*}),z *)}(-\bar x, 0)} = (-h^*)((0,y^*_{z^*}),z^*)\\
%   \Leftrightarrow\quad & h_{z^*}(\bar x,0)=  \cl\bigcup\limits_{x\in X} \Phi_{z^*}(x,0).
%  \end{align*}
\end{theorem}
\begin{proof}
For all $z^*\in C^-\setminus\cb{0}$, $\varphi_{(f,z^*)}$ is convex since $f$ is convex. 

If $\varphi_{(f,z^*)}$ is proper then we can apply Theorem \ref{ThmScalarFD} and obtain
\[
\exists y^*_{z^*}\in Y^*:\quad  \inf\limits_{x\in X} \varphi_{(f,z^*)}(x,0) =-(\varphi_{(f,z^*)})^*(0,y^*_{z^*}).
\]
By Lemma \ref{LemInf} (ii) we get $\varphi_{(f_X,z^*)}(0)=-(\varphi_{(f,z^*)})^*(0,y^*_{z^*})$.

If $\varphi_{(f,z^*)}$ is not proper then there are $x\in X$, $y\in Y$ such that 
$\varphi_{(f,z^*)}(x,y)=-\infty$ (since $(x_0,0)\in\dom\varphi_{(f,z^*)}$). Hence $\varphi_{(f_X,z^*)}(y)=-\infty$ and $\varphi_{(f_X,z^*)}(0)=-\infty$, too, since 
$0\in\Int\of{\dom \varphi_{(f,z^*)}(x_0,\cdot)}\subseteq\Int\of{\dom \varphi_{(f_X,z^*)}}$ due to the upper semicontinuity assumption. Consequently, we have 
\[
\cb{z\in Z\st \varphi_{(f_X,z^*)}(0)\leq -z^*(z)}=Z
\]
in this case.

Since $f_X(0)$ is convex we have
\begin{align*}
f_X(0)&=\bigcap_{z^*\in C^-\setminus\cb{0}}\cb{z\in Z\st \varphi_{(f_X,z^*)}(0)\leq -z^*(z)}\\
&=\bigcap_{\substack{z^*\in C^-\setminus\cb{0}\\ \varphi_{(f,z^*)} \text{ is proper}} }\cb{z\in Z\st \varphi_{(f_X,z^*)}(0)\leq -z^*(z)}\\
&= \bigcap\limits_{\substack{z^*\in C^-\setminus\cb{0}\\ \varphi_{(f,z^*)} \text{ is proper}} } \cb{z\in Z\st -(\varphi_{(f,z^*)})^*(0,y^*_{z^*})\leq -z^*(z) }\\
&= \bigcap\limits_{\substack{z^*\in C^-\setminus\cb{0}\\ \varphi_{(f,z^*)} \text{ is proper}} } (-f^*)((0,y^*_{z^*}),z^*)\\
&\supseteq \bigcap_{\substack{z^*\in C^-\setminus\cb{0} \\ y^*\in Y^*}} (-f^*)((0,y^*),z^*)\\
&\supseteq f_X(0),
\end{align*}
where the last inclusion follows from Lemma \ref{lem:1}.
\end{proof}

\begin{remark}
A similar result was already proven in \cite[Theorem 5.6 (b)]{Schrage11}. There the positive conjugate  was used instead of the negative conjugate, but the expression $H(z^*){-^{\negmedspace\centerdot\,}}f^*(0,y^*,z^*)$ in \cite{Schrage11} coincides with $(-f^*)((0,y^*),z^*)$ above.

The main difference between the two results is that we could weaken the regularity condition due to our considerations in Section \ref{sec:CCC} and we do not need the assumption that $f_X$ is $z^*$-proper for all $z^*\in C^-\setminus\cb{0}$.
\end{remark}

\begin{remark}
The existence of the family $\cb{y^*_{z^*} \st z^*\in C^-\setminus\cb{0},\; \varphi_{(f,z^*)} \text{ is proper}}$ replaces the attainment of the supremum in formula \eqref{eq:max}. In the set-valued case the infimum is no longer attained in a single point but in a set of points. In fact, the set 
\[
\cb{(y^*_{z^*},z^*) \st z^*\in C^-\setminus\cb{0},\; \varphi_{(f,z^*)} \text{ is proper}}
\]
is a solution of the optimization problem
\[
\text{maximize } (-f^*)((0,\cdot),\cdot) : Y^*\times Z^* \to \F(Z,C) \text{ w.r.t. } \supseteq \text{ over }  Y^*\times C^-\setminus\cb{0}
\]
in the sense of \cite{HL2011}.
\end{remark}

\section*{Acknowledgement}

We would like to thank Andreas Hamel and an anonymous referee for some helpful suggestions.

\bibliographystyle{abbrv}
\bibliography{CSVM}

\begin{thebibliography}{10}

\bibitem{AliprantisBorder}
C.~D. Aliprantis and K.~C. Border.
\newblock {\em Infinite-dimensional analysis}, volume~4 of {\em Studies in
  Economic Theory}.
\newblock Springer-Verlag, Berlin, 1994.
\newblock A hitchhiker's guide.

\bibitem{AubinCellina}
J.-P. Aubin and A.~Cellina.
\newblock {\em Differential inclusions}, volume 264 of {\em Grundlehren der
  Mathematischen Wissenschaften [Fundamental Principles of Mathematical
  Sciences]}.
\newblock Springer-Verlag, Berlin, 1984.
\newblock Set-valued maps and viability theory.

\bibitem{AubinFrankowska}
J.-P. Aubin and H.~Frankowska.
\newblock {\em Set-valued analysis}, volume~2 of {\em Systems \& Control:
  Foundations \& Applications}.
\newblock Birkh\"auser Boston Inc., Boston, MA, 1990.

\bibitem{Beer87}
G.~Beer.
\newblock Lattice-semicontinuous mappings and their application.
\newblock {\em Houston J. Math.}, 13(3):303--318, 1987.

\bibitem{Beer}
G.~Beer.
\newblock {\em Topologies on closed and closed convex sets}, volume 268 of {\em
  Mathematics and its Applications}.
\newblock Kluwer Academic Publishers Group, Dordrecht, 1993.

\bibitem{Berge}
C.~Berge.
\newblock {\em Topological spaces}.
\newblock Dover Publications Inc., Mineola, NY, 1997.
\newblock Including a treatment of multi-valued functions, vector spaces and
  convexity, Translated from the French original by E. M. Patterson, Reprint of
  the 1963 translation.

\bibitem{Deimling}
K.~Deimling.
\newblock {\em Multivalued differential equations}, volume~1 of {\em de Gruyter
  Series in Nonlinear Analysis and Applications}.
\newblock Walter de Gruyter \& Co., Berlin, 1992.

\bibitem{GRTZ}
A.~G{\"o}pfert, H.~Riahi, C.~Tammer, and C.~Z{\u{a}}linescu.
\newblock {\em Variational methods in partially ordered spaces}.
\newblock CMS Books in Mathematics/Ouvrages de Math\'ematiques de la SMC, 17.
  Springer-Verlag, New York, 2003.

\bibitem{Hamel09}
A.~H. Hamel.
\newblock A duality theory for set-valued functions. {I}. {F}enchel conjugation
  theory.
\newblock {\em Set-Valued Var. Anal.}, 17(2):153--182, 2009.

\bibitem{Hamel11}
A.~H. Hamel.
\newblock A {F}enchel-{R}ockafellar duality theorem for set-valued
  optimization.
\newblock {\em Optimization}, 60(8-9):1023--1043, 2011.

\bibitem{HH2010}
A.~H. Hamel and F.~Heyde.
\newblock Duality for set-valued measures of risk.
\newblock {\em SIAM J. Financial Math.}, 1:66--95, 2010.

\bibitem{HHR2011}
A.~H. Hamel, F.~Heyde, and B.~Rudloff.
\newblock Set-valued risk measures for conical market models.
\newblock {\em Math. Financ. Econ.}, 5(1):1--28, 2011.

\bibitem{HL2011}
F.~Heyde and A.~L\"ohne.
\newblock Solution concepts in vector optimization. a fresh look at an old
  story.
\newblock {\em Optimization}, 60(12):1421--1440, 2011.

\bibitem{Jahn04}
J.~Jahn.
\newblock {\em {Vector Optimization. Theory, applications, and extensions}}.
\newblock Springer-Verlag, Berlin, 2004.

\bibitem{Joly}
J.-L. Joly.
\newblock Une famille de topologies sur l'ensemble des fonctions convexes pour
  lesquelles la polarit\'e est bicontinue.
\newblock {\em J. Math. Pures Appl. (9)}, 52:421--441 (1974), 1973.

\bibitem{BookLoehne}
A.~L\"ohne.
\newblock {\em Vector optimization with infimum and supremum}.
\newblock Springer-Verlag, Berlin, 2011.

\bibitem{PenotThera82}
J.-P. Penot and M.~Th{\'e}ra.
\newblock Semicontinuous mappings in general topology.
\newblock {\em Arch. Math. (Basel)}, 38(2):158--166, 1982.

\bibitem{Peressini}
A.~L. Peressini.
\newblock {\em Ordered topological vector spaces}.
\newblock Harper \& Row Publishers, New York, 1967.

\bibitem{RoWe98}
R.~T. Rockafellar and R.~J.-B. Wets.
\newblock {\em Variational analysis}, volume 317 of {\em Grundlehren der
  Mathematischen Wissenschaften [Fundamental Principles of Mathematical
  Sciences]}.
\newblock Springer-Verlag, Berlin, 1998.

\bibitem{DissSchrage}
C.~Schrage.
\newblock {\em Set-Valued Convex Analysis}.
\newblock PhD thesis, Martin-Luther-Universit\"at Halle-Wittenberg, 2009.

\bibitem{Schrage11}
C.~Schrage.
\newblock Scalar representation and conjugation of set-valued functions.
\newblock submitted, 2011.

\bibitem{VeronaVerona90}
A.~Verona and M.~E. Verona.
\newblock Locally efficient monotone operators.
\newblock {\em Proc. Amer. Math. Soc.}, 109(1):195--204, 1990.

\bibitem{WongNg}
Y.~C. Wong and K.~F. Ng.
\newblock {\em Partially ordered topological vector spaces}.
\newblock Clarendon Press, Oxford, 1973.
\newblock Oxford Mathematical Monographs.

\bibitem{Zalinescu}
C.~Z{\u{a}}linescu.
\newblock {\em Convex analysis in general vector spaces}.
\newblock World Scientific Publishing Co. Inc., River Edge, NJ, 2002.

\end{thebibliography}
%\begin{thebibliography}{99}
%\bibitem{GRTZ}
%	A. G\"opfert, H. Riahi, C. Tammer, C. Zalinescu: {\it Variational Methods in Partially Ordered Spaces.} Springer, New York,
%	2003 
%\bibitem{HL2011}
%	F. Heyde, A. L\"ohne: Solution concepts in vector optimization: a fresh look at an old story , {\it Optimization}, 
%	to appear, DOI: 10.1080/02331931003665108, (2010)
%\end{thebibliography}
\end{document}